\documentclass[12pt]{amsart}
\usepackage{amssymb,color}
\numberwithin{equation}{section} 

\usepackage{graphicx}

\newcommand{\bea}{\begin{eqnarray}}
\newcommand{\eea}{\end{eqnarray}}
\newcommand{\ba}{\begin{array}}
\newcommand{\ea}{\end{array}}
\newcommand{\edc}{\end{document}}
\newcommand{\bc}{\begin{center}}
\newcommand{\ec}{\end{center}}
\newcommand{\be}{\begin{equation}}
\newcommand{\ee}{\end{equation}}

\newcommand{\dsf}{\displaystyle\frac}

\def\ca{{\mathcal A}}

\def\ce{{\mathcal E}}
\def\cf{{\mathcal F}}

\def\bc{{\mathbb C}}

\def\bn{{\mathbb N}}

\def\bq{{\mathbb Q}}
\def\br{{\mathbb R}}

\def\bz{{\mathbb Z}}

\def\a{\alpha}
\def\b{\beta}
  \def\G{\Gamma}
  \def\D{\Delta}

\def\l{\lambda}

\def\s{\sigma}

\def\v{\phi}

\def\u{{\mathbf{u}}}
\def\v{{\mathbf{v}}}
\def\xb{{\mathbf{x}}}
\def\fb{{\mathbf{f}}}

\def\yb{{\mathbf{y}}}

\newtheorem{thm}{Theorem}[section]
\newtheorem{lem}[thm]{Lemma}

\newtheorem{prop}[thm]{Proposition}

\theoremstyle{remark}
\newtheorem{rem}{Remark}[section]

\begin{document}
\title[Recurrence equations]
{on non-Archimedean recurrence equations and their applications}
\author{Farrukh Mukhamedov}
\address{Farrukh Mukhamedov\\
 Department of Computational \& Theoretical Sciences\\
Faculty of Science, International Islamic University Malaysia\\
P.O. Box, 141, 25710, Kuantan\\
Pahang, Malaysia} \email{{\tt far75m@yandex.ru} {\tt farrukh\_m@iium.edu.my}}
\author{Hasan Ak\i n}
\address{Hasan Ak\i n, Department of Mathematics, Faculty of Education,
 Zirve University, Kizilhisar Campus, Gaziantep, TR27260, Turkey}
\email{{\tt hasanakin69@gmail.com}}

\begin{abstract}
In the present paper we study stability of recurrence equations
(which in particular case contain a dynamics of rational functions)
generated by contractive functions defined on an arbitrary
non-Archimedean algebra. Moreover, multirecurrence equations are
considered. We also investigate reverse recurrence equations which
have application in the study of $p$-adic Gibbs measures. Note that
our results also provide the existence of unique solutions of
nonlinear functional equations as well.

\end{abstract}

\maketitle

\section{introduction}

In this paper we deal with regulation properties of discrete
dynamical systems defined over non-archimedean algebars. Note that
the interest in such systems and in the ways in which they can be
applied has been rapidly increasing during the last couple of
decades (see, e.g., \cite{AKh,Sil}). An example of non-archimedean
algebras is a field of $p$-adic numbers (see \cite{Es} for more
examples). We stress that applications of $p$-adic numbers in
$p$-adic mathematical physics \cite{MP, V1, V2}, quantum mechanics
and many others \cite{AKS,DKKV,Kh2,VVZ} stimulated increasing
interest in the study of $p$-adic dynamical systems.

On the other hand, the study of $p$-adic dynamical systems arises in
Diophantine geometry in the constructions of canonical heights, used
for counting rational points on algebraic vertices over a number
field, as in \cite{CS}. In \cite{BM} dynamical systems (not only
monomial) over finite field extensions of the $p$-adic numbers were
considered. Other studies of non-Archimedean dynamics in the
neighborhood of a periodic and of the counting of periodic points
over global fields using local fields appeared in
\cite{Fan,HY,KhN,L,LP,QWYY}.  It is known that the analytic
functions play important roles in complex analysis. In the
non-Archimedean analysis the rational functions play a role similar
to that of analytic functions in complex analysis \cite{Es}.
Therefore, there naturally arises a question as regards the study
the dynamics of these functions in the mentioned setting. In
\cite{B1,RL} a general theory of $p$-adic rational dynamical systems
over complex $p$-adic filed $\bc_p$ has been developed. Certain
rational $p$-adic dynamical systems were investigated in
\cite{ARS,KM1,M,MR1}, which appear from problems of $p$-adic Gibbs
measures \cite{KM,Mq2,MR1,MR2}. In these investigations it is
important to know the regularity or stability of the trajectories of
rational dynamical systems.

In the present paper we are going to study stability of recurrence
equations (which in particular case contain a dynamics of rational
functions) generated by contractive functions defined on an
arbitrary non-Archimedean algebra. It is also considered and studied
multirecurrence equations. Note that in \cite{vdP} certian type of
$p$-adic difference equations has been studied. In section 4 we
investigate reverse recurrence equations which have application in
the study of $p$-adic Gibbs measures. In the last section 5 we
provide applications of the main results. Note that our results also
provide the existence of unique solutions of nonlinear functional
equations as well.

\section{Preliminaries}

Let $K$ be a field with a non-Archimedean norm $|\cdot|$, i.e. for
all $x,y\in K$ one has
\begin{itemize}
\item[1.] $|x|\geq 0$ and $|x|=0$ implies $x=0$;

\item[2.] $|xy|=|x||y|$;

\item[3.] $|x+y|\leq\max\{|x|,|y|\}$.
\end{itemize}

An example of such kind of field can be considered the $p$-adic
field $\bq_p$. Namely, for a fixed prime $p$, the set $\bq_p$ is
defined as a completion of the rational numbers $\bq$ with respect
to the norm $|\cdot|_p:\bq\to\br$ given by
\begin{eqnarray}
|x|_p=\left\{
\begin{array}{c}
  p^{-r} \ x\neq 0,\\
  0,\ \quad x=0,
\end{array}
\right.
\end{eqnarray}
here, $x=p^r\frac{m}{n}$ with $r,m\in\bz,$ $n\in\bn$,
$(m,p)=(n,p)=1$. A number $r$ is called \textit{a $p-$order} of $x$
and it is denoted by $ord_p(x)=r.$ The absolute value $|\cdot|_p$ is
non- Archimedean. There are also many examples of non-Archimedean
fields (see for example \cite{Ko}).

Now let $\ca$ be a non-Archimedean Banach algebra over $K$. This
means that the norm $\|\cdot\|$ of algebra satisfies the
non-Archimedean property, i.e. $\|x+y\|\leq\max\{\|x\|,\|y\|\}$ for
any $x,y\in \ca$. There are many examples of such kind of spaces
(see \cite{Es,Ro}).

Let us consider some basic examples of non-Archimedean Banach
algebras.

1. The set
$$
K^n=\{\xb=(x_1,\dots,x_n): \ x_k\in K,\ k=1,\dots,n\}
$$
with a norm $\|\xb\|=\max|x_k|$ and usual pointwise summation and
multiplication operations, is a non-Archimedean Banach algebra.

2. Let
$$
c_0=\{\xb=(x_n): \ x_n\in K,\ x_n\to 0\}.
$$
The defined set is endowed with usual pointwise summation and
multiplication operations. Put $\|\xb\|=\max|x_k|$, then $c_0$ is a
non-Archimedean Banach algebra.\\

In what follows, by $\ca$ we denote a non-Archimedean Banach
algebra.

There is a nice characterization of Cauchy sequence in
non-Archimedean spaces.

\begin{prop}\cite{Ko}\label{Ca}
 A sequence $\{x_n\}$ in  $\ca$ is a Cauchy sequence
 with respect to the norm $\|\cdot\|$ if and only if $\|x_{n+1}-x_{n}\|\to 0$
 as $n\to\infty$.
 \end{prop}

Denote
\begin{eqnarray*}
&&B(a,r)=\{x\in \ca : \|x-a\|< r\}, \ \ \bar B(a,r)=\{x\in\ca :
\|x-a\|\leq r\},\\
&& S(a,r)=\{x\in \ca : \|x-a\|=r\},
\end{eqnarray*}
where $a\in\ca$, $r>0$.

In what follows, we will use the following

\begin{lem}[\cite{KMM}]\label{pr} Let $\{a_i\}_{i=1}^n,\{b_i\}_{i=1}^n\subset \ca$ such that
$\|a_i\|\leq 1$, $\|b_i\|\leq 1$, $i=1,\dots,n$, then
\begin{equation*}
\bigg\|\prod_{i=1}^{n}a_i-\prod_{i=1}^n b_i\bigg\|\leq \max_{i\leq
i\leq n}\{\|a_i-b_i\|\}
\end{equation*}
\end{lem}

Note that the basics of non-Archimedean analysis are explained in
\cite{S,R}.

\section{A recurrence equations}

Let $\ca$ be a non-Archimedean Banach algebra and assume that
$C\subset \bar B(0,1)$ be a closed set. A mapping $f:C^m\to C$ is
called \textit{contractive}, if there is a constant $\a_f\in [0,1)$
such that
\begin{eqnarray}\label{C}
\|f(\xb)-f(\yb)\|\leq \a_f\max_{i\leq k\leq m}\|x_k-y_k\| \ \
\textrm{for all}\ \  \xb=(x_i),\yb=(y_i)\in C^m.
\end{eqnarray}

Note that if the function $f$ does not depend on some variable
$x_k$, then such a variable will be absent in the right hand side of
\eqref{C}.

Now assume that we are given several collections
$\{f_{i}^{(k)}\}_{i=1}^N$, $k=1,\dots,M$ of contractive mappings
defined on $C^m$. Let $(\ell^{(k)}_1,\dots,\ell^{(k)}_N)$ such that
$\ell^{(k)}_1=0$, $2\leq \ell^{(k)}_i-\ell^{(k)}_{i-1}\leq m-1$,
$i=2,\dots,N$, $k=1,\dots,M$.

Denote $L:=\max\{\ell^{(k)}_m: \ 1\leq k\leq M\}+m$. Take any
initial points $\{x_1,\dots,x_{L}\}\subset C$, and consider the
following sequence $\{x_n\}$ defined by the recurrence relations:
\begin{eqnarray}\label{R1}
x_{n+L}=\sum_{k=1}^M \prod_{i=1}^N
f_{i}^{(k)}(x_{n+\ell^{(k)}_i},\dots,x_{n+\ell^{(k)}_i+m-1}), \ \ \
n\in\bn.
\end{eqnarray}

\begin{lem}\label{Rl1}
Let $\{f_{i}^{(k)}\}_{i=1}^N$, $k=1,\dots,M$ be collections of
contractive mappings defined on $C^m$ (where $C\subset \bar
B(0,1)$). Then for any initial points $\{x_1,\dots,x_{m+L}\}\subset
C$ the sequence $\{x_n\}$ defined by \eqref{R1} is convergent.
\end{lem}

\begin{proof} To prove the lemma it is enough to show that $\{x_n\}$ is a
Cauchy sequence. Due to Proposition \ref{Ca} we need to establish
$\|x_{n+1}-x_{n}\|\to 0$ as $n\to\infty$. Let us first denote
$$
\a=\max_{k,i}\a_{f_{i}^{(k)}}.
$$
From the contractivity of the functions $f_{i}^{(k)}$ we conclude
that $0<\a<1$.

Now from \eqref{R1}, $\|f_{i}^{(k)}(\xb)\|\leq 1$, $\xb\in C^m$ and
using Lemma \ref{pr} one finds
\begin{eqnarray*}
\|x_{n+L+1}-x_{n+L}\|&\leq& \max_{1\leq k\leq M}\bigg\|
\prod_{i=1}^N
f_{i}^{(k)}(x_{n+1+\ell^{(k)}_i},\dots,x_{n+1+\ell^{(k)}_i+m-1})-\\[2mm]
&&\qquad \ \  \prod_{i=1}^N
f_{i}^{(k)}(x_{n+\ell^{(k)}_i},\dots,x_{n+\ell^{(k)}_i+m-1})\bigg\|\nonumber\\[3mm]
&\leq&\max_{1\leq k\leq M \atop 1\leq i\leq N}
\bigg\|f_{i}^{(k)}(x_{n+1+\ell^{(k)}_i},\dots,x_{n+1+\ell^{(k)}_i+m-1})-\\[2mm]
&&\qquad \ \
f_{i}^{(k)}(x_{n+\ell^{(k)}_i},\dots,x_{n+\ell^{(k)}_i+m-1})\bigg\|\nonumber\\[3mm]
&\leq&\a\bigg(\max_{1\leq i\leq
N}\big\|x_{n+1+i+\ell^{(k)}_i}-x_{n+i+\ell^{(k)}_i}\big\|\bigg)\nonumber\\[3mm]
&&\cdots\cdots\cdots\\[2mm]
&\leq&\a^{n-L}.
\end{eqnarray*}
The last inequality yields that $\|x_{n+1}-x_{n}\|\to 0$ as
$n\to\infty$. Due to closedness of $C$ we conclude that the limiting
element belongs to $C$. This completes the proof.
\end{proof}

\begin{lem}\label{Rl2}
Let $\{f_{i}^{(k)}\}_{i=1}^N$, $k=1,\dots,M$ be collections of
contractive functions defined on $C^m$ (where $C\subset \bar
B(0,1)$).  Take any two colloctions of initial points, i.e.
$\{x_1,\dots,x_{m+L}\}\subset C$ and $\{y_1,\dots,y_{m+L}\}\subset
C$. Then for the corresponding sequences $\{x_n\}$ and $\{y_n\}$,
defined by \eqref{R1}, one has $\|x_n-y_n\|\to 0$ as $n\to\infty$.
\end{lem}

\begin{proof}
From \eqref{R1}, $\|f_{i}^{(k)}(x)\|\leq 1$, $\xb\in C^m$ and using
Lemma \ref{pr} one finds
\begin{eqnarray*}
\|x_{n+L}-y_{n+L}\|&\leq&\max_{1\leq k\leq M}\bigg\| \prod_{i=1}^N
f_{i}^{(k)}(x_{n+\ell^{(k)}_i},\dots,x_{n+\ell^{(k)}_i+m-1})-\\[2mm]
&&\qquad \ \  \prod_{i=1}^N
f_{i}^{(k)}(y_{n+\ell^{(k)}_i},\dots,y_{n+\ell^{(k)}_i+m-1})\bigg\|\nonumber\\[3mm]
&\leq&\max_{1\leq k\leq M \atop 1\leq i\leq N}
\bigg\|f_{i}^{(k)}(x_{n+\ell^{(k)}_i},\dots,x_{n+\ell^{(k)}_i+m-1})-\\[2mm]
&&\qquad \ \
f_{i}^{(k)}(y_{n+\ell^{(k)}_i},\dots,y_{n+\ell^{(k)}_i+m-1})\bigg\|\nonumber\\[3mm]
&\leq&\a\bigg(\max_{1\leq i\leq
N}\big\|x_{n+i+\ell^{(k)}_i}-y_{n+i+\ell^{(k)}_i}\big\|\bigg)\nonumber\\[3mm]
&&\cdots\cdots\cdots\\[2mm]
&\leq&\a^{n-L}.
\end{eqnarray*}
The last inequality implies that $\|x_{n}-y_{n}\|\to 0$ as
$n\to\infty$. The proof is complete.
\end{proof}

From these lemmas we infer the following

\begin{thm}\label{Rt1}
Let $\{f_{i}^{(k)}\}_{i=1}^N$, $k=1,\dots,M$ be collections of
contractive functions defined on $C^m$ (where $C\subset \bar
B(0,1)$).  Then there is $x_*\in C$ such that for any initial points
$\{x_1,\dots,x_L\}\subset C$ the sequence $\{x_n\}$ defined by
\eqref{R1} converges to $x_*$. Moreover, one has
$$
\|x_{n+L}-x_*\|\leq \a^{n} \ \ \ \textrm{for all} \ \ n\in\bn,
$$
where
$$
\a=\max_{k,i}\a_{f_{i}^{(k)}}.
$$
\end{thm}

\begin{rem} From the last theorem we infer that the sequence \eqref{R1} defines
a unique solution (belonging to the set $C$) of the equation
\begin{eqnarray}\label{eR}
x=\sum_{k=1}^M \prod_{i=1}^Nf_{i}^{(k)}(x,\dots,x).
\end{eqnarray}
\end{rem}

\begin{rem} If $f(\xb)$ a contractive mapping on $C^m$, then Theorem \ref{Rt1} yields
that for any $N>1$ the equation
\begin{eqnarray*}\label{eR}
x=\big(f(x,\dots,x)\big)^N.
\end{eqnarray*}
has a unique solution belonging to $C$. More conctere examples will
be given in the final section.
\end{rem}

Mow let us consider multisequence case.

As before $\ca$ denotes a non-Archimedean Banach algebra and assume
that $C\subset \bar B(0,1)$ be a closed set. Suppose that we are
given several collections $\{F_{1}^{(k)},F_{2}^{(k)}\}_{k=1}^{N_1}$,
$\{G_{1}^{(k)},G_{2}^{(k)}\}_{k=1}^{N_2}$,
$\{H_{1}^{(k)},H_{2}^{(k)}\}_{k=1}^{N_3}$ of contractive mappings
defined on $C^2$.

 Take any initial points
$\{x_1,y_1,z_1\}\subset C$, and consider the following sequences
$\{x_n\}$, $\{y_n\}$, $\{z_n\}$ defined by the recurrence relations:
\begin{eqnarray}\label{nR1}
\left\{ \begin{array}{lll} && x_{n+1}=\sum_{k=1}^{N_1}
F_{1}^{(k)}(x_{n},y_{n})F_{2}^{(k)}(y_{n},z_{n})\\[2mm]
&& y_{n+1}=\sum_{k=1}^{N_2}
G_{1}^{(k)}(x_{n+1},y_{n})G_{2}^{(k)}(y_{n},z_{n})\\[2mm]&&
z_{n+1}=\sum_{k=1}^{N_3}
H_{1}^{(k)}(x_{n+1},y_{n+1})H_{2}^{(k)}(y_{n+1},z_{n})\\[2mm]
\end{array}
\right.
\end{eqnarray}

\begin{thm}\label{nRt1}
Let $\{F_{1}^{(k)},F_{2}^{(k)}\}_{k=1}^{N_1}$,
$\{G_{1}^{(k)},G_{2}^{(k)}\}_{k=1}^{N_2}$,
$\{H_{1}^{(k)},H_{2}^{(k)}\}_{k=1}^{N_3}$  be collections of
contractive mappings defined on $C^2$ (where $C\subset \bar
B(0,1)$). Then for any initial points $\{x_1,y_1,z_1\}\subset C$ the
sequences $\{x_n\}$, $\{y_n\}$, $\{z_n\}$ defined by \eqref{nR1} are
convergent. Moreover, the limit does not depend on initial
condtions.
\end{thm}

\begin{proof} First we prove that each sequence is a Cauchy sequence.
Let us denote
\begin{eqnarray*}
&&
d_n=\max\{\|x_{n+1}-x_{n}\|,\|y_{n+1}-y_{n}\|,\|z_{n+1}-z_{n}\|\}, \
\\[2mm]
&&\a=\max_{k,i}\{\a_{F_{i}^{(k)}},\a_{G_{i}^{(k)}},\a_{H_{i}^{(k)}}\}.
\end{eqnarray*}

Due to condition we have that $0<\a<1$.

Now from \eqref{nR1}, $\|f_{i}^{(k)}(\xb)\|\leq 1$ and using Lemma
\ref{pr} one finds
\begin{eqnarray}\label{nRt11}
\|x_{n+1}-x_{n}\|&\leq& \max_{1\leq k\leq N_1}\bigg\|
F_{1}^{(k)}(x_{n},y_{n})F_{2}^{(k)}(y_{n},z_{n})-F_{1}^{(k)}(x_{n-1},y_{n-1})F_{2}^{(k)}(y_{n-1},
z_{n-1})\bigg\|\nonumber \\[2mm]
&\leq & \max_{1\leq k\leq N_1}\max\{\big\|
F_{1}^{(k)}(x_{n},y_{n})-F_{1}^{(k)}(x_{n-1},y_{n-1})\big\|,
\big\|F_{2}^{(k)}(y_{n},z_{n})-F_{2}^{(k)}(y_{n-1}, z_{n-1})\big\|\}
\nonumber\\[3mm]
&\leq&\a\max\{\|x_{n}-x_{n-1}\|,\|y_{n}-y_{n-1}\|,\|z_{n}-z_{n-1}\|\}.
\end{eqnarray}
Using the same argument we obtain
\begin{eqnarray}\label{nRt12}
&&\|y_{n+1}-y_{n}\|\leq
\a\max\{\|x_{n+1}-x_{n}\|,\|y_{n}-y_{n-1}\|,\|z_{n}-z_{n-1}\|\},\\[2mm]
\label{nRt13} &&\|z_{n+1}-z_{n}\|\leq
\a\max\{\|x_{n+1}-x_{n}\|,\|y_{n+1}-y_{n}\|,\|z_{n}-z_{n-1}\|\}.
\end{eqnarray}

Hence from \eqref{nRt11}-\eqref{nRt13} one finds
\begin{eqnarray}\label{nRt14}
d_{n+1}\leq \a d_n
\end{eqnarray}
for all $n\in\bn$. This means that $d_n\to 0$ as $n\to\infty$. Due
to Proposition \ref{Ca} the sequences are Cauchy. The closedness of
$C$ yields that the limiting elements belongs to $C$, i.e. $x_n\to
x_*$, $y_n\to y_*$, $z_n\to z_*$, where $x_*, y_*,z_*\in C$.

The uniqueness of the limiting elements can by proved by the same
argument as the proof of Lemma \ref{Rl2}. This completes the proof.
\end{proof}

\begin{rem} From Theorem \ref{nRt1} we conclude that the sequences \eqref{nR1} define
a unique solution (belonging to the set $C$) of the system of
equations
\begin{eqnarray}\label{nR2}
\left\{ \begin{array}{lll} && x=\sum_{k=1}^{N_1}
F_{1}^{(k)}(x,y)F_{2}^{(k)}(y,z)\\[2mm]
&& y=\sum_{k=1}^{N_2}
G_{1}^{(k)}(x,y)G_{2}^{(k)}(y,z)\\[2mm]&&
z=\sum_{k=1}^{N_3}
H_{1}^{(k)}(x,y)H_{2}^{(k)}(y,z)\\[2mm]
\end{array}
\right.
\end{eqnarray}
Note that a'priori the existence of the solution of \eqref{nR2} is
not obvious. Moreover, the proved Theorem \ref{nR1} allows to find
solutions of functional equations, when one takes instead of $\ca$
the algebra of analytic functions. In \cite{EOY} polynomial
functional equations have been investigated over $p$-adic analytic
functions.
\end{rem}

\begin{rem}
We stress that by modifying \eqref{nR1} for arbirary number of
sequmnces, similar kind of results can be proved by means of the
same technuque as in the proof of Theorem \ref{nRt1}.
\end{rem}

\section{A reverse recurrence equations}

In this section we consider a reverse recurrence relations to
\eqref{R1}. To define it, we need some prelimenary notions about a
$k$-ary trees.

Let $(V,L)$ be a graph, here $V$ is the set of vertices and $L$ is
the set of edges.  A pair $G_k=(V,L)$ is called \textit{ $k$-ary
tree} if it has a root $x^0$ in which each vertex has no more than
$k$ edges. If in a $k$-ary tree each vertex has exactly $k$ edges,
then such a tree is called \textit{Cayley tree}. The vertices $x$
and $y$ are called {\it nearest neighbors} and they are denoted by
$l=<x,y>$ if there exists an edge connecting them. A collection of
the pairs $<x,x_1>,\dots,<x_{d-1},y>$ is called a {\it path} from
the point $x$ to the point $y$. The distance $d(x,y), x,y\in V$, on
the tree, is the length of the shortest path from $x$ to $y$.

Recall a coordinate structure in $G_k$:  every vertex $x$ (except
for $x^0$) of $\G_k$ has coordinates $(i_1,\dots,i_n)$, here
$i_m\in\{1,\dots,k\}$, $1\leq m\leq n$ and for the vertex $x^0$ we
put $(0)$.  Namely, the symbol $(0)$ constitutes level 0, and the
sites $(i_1,\dots,i_n)$ form level $n$ ( i.e. $d(x^0,x)=n$) of the
lattice.


For $x\in G_k$, $x=(i_1,\dots,i_n)$ denote
\begin{equation}\label{S(x)}
 S(x)=\{(x,i):\ 1\leq
i\leq k_x\},
\end{equation}
here $(x,i)$ means that $(i_1,\dots,i_n,i)$. This set is called a
set of {\it direct successors} of $x$.

Let $\ca$ be as usual a non-Archimedean Banach algebra and $C\subset
\bar B(0,1)$. Assume that we are given a family
$\{f^{(i)}_{x,y}\}_{i=1}^M$, $<x,y>\in L$ of contractive mappinga
such that for each $<x,y>$ the function $f^{(i)}_{x,y}$ maps
$C^{k_x}$ to $C$. Now consider a function $\u: V\to C$, i.e.
$\u=((u_{x})_{x\in G_k}$ such that
\begin{eqnarray}\label{uR1}
u_{x}=\sum_{i=1}^M \prod_{y\in S(x)}
f^{(i)}_{xy}(u_{(x,1)},\dots,u_{(x,k_x)}).
\end{eqnarray}

We are interested how many functions $\u$ satisfy the equation
\eqref{uR1}.

Denote
$$
\b=\max_{1\leq i\leq M \atop <x,y>\in L}\a_{f^{(i)}_{x,y}}.
$$

\begin{thm}\label{Rt2}
Let $\{f^{(i)}_{x,y}\}_{i=1}^M$, $<x,y>\in L$ be a family of
contractive functions such that $\b<1$. Then a solution of the
equation \eqref{uR1} is not more than one.
\end{thm}

\begin{proof} If the equation \eqref{uR1} has not any solution, then nothing to prove.
Therefore, let us assume that the given equation has a solution. To
prove Theorem it is enough to show that any two solutions coincide
with each other. Namely, if $\u=(u_x,x\in V)$ and $\v=(v_x,x\in V)$
are solutions of \eqref{uR1}, then it is sufficient to establish
that for any $\varepsilon>0$ and $x\in V$ the inequality
$\|u_x-v_x\|<\varepsilon$ is valid.

Let $x\in V$ be an arbitrary vertex. Then from \eqref{uR1},
$\|f_{xy}^{(i)}(\xb)\|\leq 1$, $x\in C^{k_x}$ and using Lemma
\ref{pr} we obtain
\begin{eqnarray}\label{uR2}
\|u_x-v_x\|&\leq&\max_{1\leq i\leq M \atop y\in S(x)}\bigg\|
f^{(i)}_{xy}(u_{(x,1)},\dots,u_{(x,k_x)})-f^{(i)}_{xy}(v_{(x,1)},\dots,v_{(x,k_x)})
\bigg\|\nonumber\\[3mm]
&\leq&\b\bigg(\max_{1\leq i\leq
k_x}\big\|u_{(x,i)}-v_{(x,i)}\big\|\bigg).
\end{eqnarray}
Let us choose $n_0\in\bn$ such that $\b^{n_0}<\varepsilon$.
Therefore, iterating \eqref{uR2} $n_0$-times one gets
\begin{eqnarray}\label{uR3}
\|u_x-v_x\|\leq\b^{n_0}<\varepsilon.
\end{eqnarray}
This completes the proof.
\end{proof}

\begin{rem} We note that particalar cases of the present theorem
were proved in \cite{KM,M,MR1,MR2}. The proved theorem generalize
and extends all the known results.
\end{rem}

\section{Application}

In the section we consider the $p$-adic field $\bq_p$ ($p\geq 3)$.
Recall that the $p$-adic logarithm is defined by series
$$
\log_p(x)=\log_p(1+(x-1))=\sum_{n=1}^{\infty}(-1)^{n+1}\dsf{(x-1)^n}{n},
$$
which converges for every $x\in B(1,1)$. And $p$-adic exponential is
defined by
$$
\exp_p(x)=\sum_{n=1}^{\infty}\dsf{x^n}{n!},
$$
which converges for every $x\in B(0,p^{-1/(p-1)})$.

\begin{lem}\label{21} \cite{Ko} Let $x\in
B(0,p^{-1/(p-1)})$ then we have $$ |\exp_p(x)|_p=1,\ \ \
|\exp_p(x)-1|_p=|x|_p<1, \ \ |\log_p(1+x)|_p=|x|_p<p^{-1/(p-1)} $$
and $$ \log_p(\exp_p(x))=x, \ \ \exp_p(\log_p(1+x))=1+x. $$
\end{lem}

\begin{rem} Note that, in general, the logarithm and the exponential
functions can be defined over the field $K$ with $char(K)=0$ (see
\cite{S}).
\end{rem}

Denote
\begin{equation}\label{Exp}
\ce_p=\{x\in\bq_p: \ |x|_p=1, \ \ |x-1|_p\leq 1/p\}.
\end{equation}

Note that from Lemma \ref{21} one concludes that if $x\in\ce_p$,
then there is an element $h\in B(0,p^{-1/(p-1)})$ such that
$x=\exp_p(h)$. Therfore, for any $x,y\in\ce_p$ one gets
$xy\in\ce_p$.

{\bf 1.}  Assume that $\ca=\bq_p$ and $C=\ce_p$. Let us consider a
non-linear function:
\begin{equation}\label{c1}
f(x,y)=\frac{axy+b(x+y)+c}{a_1xy+b_1(x+y)+c_1},
\end{equation}
where $a,a_1,b,b_1,c,c_1\in \ce_p$.

\begin{prop}\label{aa} Let $f$ be given by \eqref{c1}. Then one has
\begin{itemize}
\item[(i)] $f(x,y)\in\ce_p$ for any $x,y\in\ce_p$;
\item[(ii)] the function $f$ is contractive.
\end{itemize}
\end{prop}

\begin{proof} (i). Take any $x,y\in\ce_p$. Then one can see that
\begin{eqnarray}\label{cc}
|axy+b(x+y)+c|_p=|axy-1+b(x-1+y-1)+2(b-1)+c-1+4|_p=1,
\end{eqnarray}
since $a,a_1,b,b_1,c,c_1\in \ce_p$. Similarly, we get
\begin{eqnarray}\label{cc1}
|a_1 xy+b_(x+y)+c_1|_p=1.
\end{eqnarray}
Therefore, $|f(x,y)|_p=1$.
Using the same manner from
$$
|f(x,y)-1|_p=|(a-a_1)xy+(b-b_1)(x+y)+c-c_1|_p\leq \frac{1}{p}
$$
we find that $f(x,y)\in\ce_p$.

(ii) Now take any $(x,y),(x_1,y_1)\in \ce_p\times\ce_p$. Then using
\eqref{cc},\eqref{cc1} one finds
\begin{eqnarray}\label{cc2}
|f(x,y)-f(x_1,y_1)|_p=|\D_1(x-x_1)+\D_2(y-y_1)|_p
\end{eqnarray}
where
\begin{eqnarray*}\label{cc3}
\D_1&=&(ab_1-a_1b)yy_1+(ac_1-a_1c)y_1+c_1b-cb_1,\\[2mm]
\D_2&=&(ab_1-a_1b)xx_1+(ac_1-a_1c)x+c_1b-cb_1.
\end{eqnarray*}
It is easy to see that $|\D_1|_p\leq 1/p$, $|\D_2|_p\leq 1/p$.
Hence, from \eqref{cc2} we have
\begin{eqnarray}\label{cc2}
|f(x,y)-f(x_1,y_1)|_p\leq \frac{1}{p}\max\{|x-x_1|_p,|y-y_1|_p\}
\end{eqnarray}
which implies the assertion.
\end{proof}

Let $G$ be a Cayley tree of order three, and consider the following
functional equation
\begin{eqnarray}\label{uR2}
u_{x}=f(u_{(x,1)},u_{(x,2)}), \ \ \
\end{eqnarray}
where $\u=((u_{x})_{x\in G}$ is unknown function and $f$ is given by
\eqref{c1}.

Then due to Theorem \ref{Rt2} the equation has  a unique solution
$u_x=u_*$. Here $u_*$ is a fixed point belonging to $\ce_p$ of the
function $f(u,u)$ which exists due to Theorem \ref{Rt1}. This
fact extends the results of the papers \cite{Khak}.\\

One can consider the following equation
\begin{eqnarray}\label{uR2}
u_{x}=\big(f(u_{(x,1)},u_{(x,2)})\big)^k.
\end{eqnarray}
This equation also has a unique solution $u_x=u_*$, where
$u_*\in\ce_p$ is a fixed point of $(f(u,u))^k$ which exists due to
Theorem \ref{Rt1}. This fact implies the main result of the paper \cite{KM}.\\

{\bf 2.} Now let us consider another kind of example.

Assume that $\ca=\bq_p$ and $C=S(0,1)$. Define a non-linear function
as follows:
\begin{equation}\label{a1}
F(x_1,\dots,x_m)=\frac{P(x_1,\dots,x_m)+C}{Q(x_1,\dots,x_m)+C_1}
\end{equation}
where
\begin{eqnarray}\label{a2}
&&P(x_1,\dots,x_m)=\sum_{i_1+\cdots +i_m=1,\atop i_k\geq 0,1\leq
k\leq m}^N
A_{i_1,\dots,i_m}x_1^{i_1}\cdots x_m^{i_m}\\[3mm]
\label{a3} &&Q(x_1,\dots,x_m)=\sum_{i_1+\cdots +i_m=1,\atop i_k\geq
0,1\leq k\leq m}^N B_{i_1,\dots,i_m}x_1^{i_1}\cdots x_m^{i_m}
\end{eqnarray}
and $A_{i_1,\dots,i_m},B_{i_1,\dots,i_m}\in B(0,1)$ and
$|C|_p=|C_1|_p=1$.

\begin{prop} Let $F$ be given by \eqref{a1}. Then one has
\begin{itemize}
\item[(i)] $F(x_1,\dots,x_m)\in S(0,1)$ for any $x_1,\dots,x_m\in S(0,1)$;
\item[(ii)] the function $F$ is contractive on $S(0,1)^m$.
\end{itemize}
\end{prop}

\begin{proof} (i). Due to $A_{i_1,\dots,i_m},B_{i_1,\dots,i_m}\in B(0,1)$ we immediately
find that $|P(x_1,\dots,x_m)|_p=|Q(x_1,\dots,x_m)|_p<1$ for any
$x_1,\dots,x_m\in S(0,1)$, which with $|C|_p=|C_1|_p=1$ implies the
assertion.

(ii) Now take any $(x_1,\dots,x_m),(y_1,\dots,y_m)\in S(0,1)^m$.
Then using \eqref{a2} and Proposition \ref{pr} one finds
\begin{eqnarray*}\label{cc2}
&&|F(x_1,\dots,x_m)-F(y_1,\dots,y_m)|_p=|C_1P(x_1,\dots,x_m)+CQ(y_1,\dots,y_m)\\[2mm]
&&\qquad \qquad +P(x_1,\dots,x_m)Q(y_1,\dots,y_m)-
C_1P(y_1,\dots,y_m)\\[2mm]
&&\qquad \qquad -CQ(x_1,\dots,x_m)- P(y_1,\dots,y_m)Q(x_1,\dots,x_m)|_p\\[2mm]
&&\qquad \qquad =\bigg|C_1\sum_{i_1,\dots,i_m}
A_{i_1,\dots,i_m}\big(x_1^{i_1}\cdots x_m^{i_m}-y_1^{i_1}\cdots
y_m^{i_m}\big)\\[3mm]
&&\qquad \qquad -C\sum_{i_1,\dots,i_m}
B_{i_1,\dots,i_m}\big(x_1^{i_1}\cdots x_m^{i_m}-y_1^{i_1}\cdots
y_m^{i_m}\big)\\[3mm]
&&+\sum_{i_1,\dots,i_m}\sum_{j_1,\dots,j_m}
A_{i_1,\dots,i_m}B_{j_1,\dots,j_m}\big(x_1^{i_1}\cdots
x_m^{i_m}y_1^{j_1}\cdots y_m^{j_m}-x_1^{j_1}\cdots
x_m^{j_m}y_1^{i_1}\cdots y_m^{i_m}\big)\bigg|_p\\[3mm]
&&\qquad \qquad \leq\frac{1}{p}\max\{|x_k-y_k|_p\},
\end{eqnarray*}
which implies the assertion.
\end{proof}

Let us consider the following sequence
$$
X_{n+2m}=F(X_n,\dots,X_{n+m})F(X_{n+1},\dots,X_{n+m})F(X_{n+m},\dots,X_{n+2m-1}),
$$
with initial condtions $X_1,\dots,X_{2m}\in S(0,1)$.

Then due to Theorem \ref{Rt1} the sequence $\{X_n\}$ converges to
$X_*\in S(0,1)$ which is a solution of the equation
$$
X=\big(F(X,\dots,X)\big)^3.
$$

{\bf 3.} Assume that $\ca=\bq_p^m$ and $C=\ce_p$, here $m+1$ is not
divisible by $p$. Define a non-linear mapping $\fb:\bq_p^m\to
\bq_p^m$ by the following formula:
\begin{equation}\label{b1}
\fb(\xb)_k=\frac{\sum_{j=1}^ma^{(k)}_jx_j+a_k}{\sum_{j=1}^mb^{(k)}_jx_j+b_k},
\ \ \xb=(x_1,\dots,x_m), \ k=1,\dots,m,
\end{equation}
where $a^{(k)}_j,b^{(k)}_j,a_k,b_k\in\ce_p$.

\begin{prop} Let $\fb$ be given by \eqref{b1}. Then one has
\begin{itemize}
\item[(i)] $\fb(\ce_p^m)\subset \ce_p^m$;
\item[(ii)] the mapping $\fb$ is contractive on $\ce_p^m$.
\end{itemize}
\end{prop}

\begin{proof} (i) Due to $a^{(k)}_j,b^{(k)}_j,a_k,b_k\in\ce_p$ and
$m+1\nmid p$ one finds
\begin{eqnarray}\label{b2}
&& \bigg|\sum_{j=1}^ma^{(k)}_jx_j+a_k\bigg|_p=
\bigg|\sum_{j=1}^m(a^{(k)}_jx_j-1)+(a_k-1)+m+1\bigg|_p=1
\end{eqnarray}
Similarly, we have
\begin{eqnarray}\label{b3}
&& \bigg|\sum_{j=1}^mb^{(k)}_jx_i+b_k\bigg|_p=1.
\end{eqnarray}
This yields that $|\fb(\xb)_k|_p=1$ for all $\xb\in\ce_p^m$,
$k\in\{1,\dots,m\}$.

Using the same argument, one can get $|\fb(\xb)_k-1|_p\leq 1/p$ for
all $\xb\in\ce_p^m$. This implies the assertion.

(ii). Now using \eqref{b2},\eqref{b3} we obtain
\begin{eqnarray}\label{b4}
|\fb(\xb)_k-\fb(\yb)_k|_p&=&
\bigg|\bigg(\sum_{j=1}^ma^{(k)}_jx_j+a_k\bigg)\bigg(\sum_{j=1}^mb^{(k)}_jy_j+b_k\bigg)\nonumber\\[3mm]
&&-\bigg(\sum_{j=1}^ma^{(k)}_jy_i+a_k\bigg)\bigg(\sum_{j=1}^mb^{(k)}_jx_j+b_k\bigg)\bigg|_p\nonumber\\[3mm]
&=&\bigg|\underbrace{\sum_{i,j=1}^ma^{(k)}_ib^{(k)}_j(x_iy_j-x_jy_i)}_I
-\underbrace{\sum_{j=1}^m(a_kb^{(k)}_j-b_ka^{(k)}_j)(x_j-y_j)}_{II}\bigg|_p.
\end{eqnarray}
Now let us rewite the expression $I$ as follows
\begin{eqnarray*}
\sum_{i,j=1}^ma^{(k)}_ib^{(k)}_j(x_iy_j-x_jy_i)&=&
\sum_{i,j=1}^ma^{(k)}_ib^{(k)}_j\big(x_i(y_j-x_j)+x_j(x_i-y_i)\big)\\[3mm]
&=&\sum_{i,j=1}^mx_i(b^{(k)}_ia^{(k)}_j-a^{(k)}_ib^{(k)}_j)\big(x_j-y_j)
\end{eqnarray*}
Therefore, we find that
$$
|I|_p\leq\frac{1}{p}\max\{|x_k-y_k|_p\}, \ \
|II|_p\leq\frac{1}{p}\max\{|x_k-y_k|_p\},
$$
Hence, the last inequllitions with \eqref{b4} implies that
$$
\|\fb(\xb)-\fb(\yb)\|\leq\frac1p\|\xb-\yb\|
$$
this proves the proposition.
\end{proof}

Let $\G_k$ be a Cayley tree  of order $k$ ($k\geq 1$). Let us
consider the functional equation
\begin{eqnarray}\label{b5}
\u_{x}=\prod_{y\in S(x)}\fb(\u_{y}), \ \ \
\end{eqnarray}
where $\u=((\u_{x})_{x\in \G_k}$ is unknown function and $\fb$ is
given by \eqref{b1}.

It is clear that the equation \eqref{b5} has a solution $\u_x=\u_*$,
where $\u_*$ is fixed point of the equation
$$
\big(\fb(\u)\big)^k=\u.
$$
Note that this solution $\u_*$ belongs to $\ce_p^m$ which follows
from Theorem \ref{Rt1}.

Now according to Theorem \ref{Rt2} we conclude that the equation
\eqref{b5} has only one solution which is  $\u_x=\u_*$. This result
can be applied to the existence and uniqueness of $p$-adic Gibbs
measure associated with $m$-state $p$-adic $\l$-model on the Cayley
tree of order $k$ (see for the definition of the model
\cite{M00}).\\

{\bf 4.} In this example, we assume that $\ca=c_0$ and $C=\bar
B(0,1)$. Define a non-linar mapping $\cf:c_0\to c_0$ as follows:
\begin{equation}\label{f1}
(\cf(\xb))_k=\l_k F_k(\xb), \ \ \xb\in c_0,
\end{equation}
where $\ell=\{\l_k\}\in c_0$ with $\|\ell\|\leq 1$, and
\begin{equation}\label{f2}
F_k(\xb)=\frac{ax_k+f_k(\xb)}{b+f_k(\xb)}.
\end{equation}
Here $a,b\in\bq_p$, $\max\{|a|_p,|b|_p\}<1$ and the functions
$\{f_k\}$ such that $|f_k(\xb)|_p=1$ for all $\xb\in\bar B(0,1)$,
$k\in\bn$ and one has
\begin{equation}\label{f3}
|f_k(\xb)-f_k(\yb)|_p\leq\|\xb-\yb\|, \ \ \textrm{for all} \ \
\xb,\yb\in \bar B(0,1).
\end{equation}

\begin{prop} Let $\cf$ be given by \eqref{f1}. Then one has
\begin{itemize}
\item[(i)] $\cf(\bar B(0,1))\subset \bar B(0,1)$;
\item[(ii)] the mapping $\cf$ is contractive on $\bar B(0,1)$.
\end{itemize}
\end{prop}

\begin{proof} (i) From $\max\{|a|_p,|b|_p\}<1$ and $|f_k(\xb)|_p=1$ for all $\xb\in\bar B(0,1)$
we immediately find that $|F_k(\xb)|_p=1$ for all $\xb\in\bar
B(0,1)$ and $k\in\bn$. Therefore, $\|\cf(\xb)\|=\|\ell\|\leq 1$,
which is the required assertion.

(ii) Take any $\xb,\yb\in \bar B(0,1)$. Then we have
\begin{eqnarray}\label{f4}
|F_k(\xb)-F_k(\yb)|&=&|ab(x_k-y_k)+a(x_kf(\yb)-y_kf_k(\xb))+b(f_k(\xb)-f_k(\yb))\|_p\nonumber
\\[2mm]
&=&\bigg|a\big(b+f_k(\xb)\big)(x_k-y_k)+(b-ax_k)(f_k(\xb)-f_k(\yb))\bigg|_p\nonumber
\\[2mm]
&\leq &\max\{|a|_p,|b|_p\}\|\xb-\yb\|.
\end{eqnarray}
Hence, from \eqref{f1} and \eqref{f4} one gets
$$
\|\cf(\xb)-\cf(\yb)\|\leq\max\{|a|_p,|b|_p\}\|\xb-\yb\|
$$
this completes the proof.
\end{proof}

Let $\G_k$ be a Cayley tree  of order $k$ ($k\geq 1$). Let us
consider the functional equation
\begin{eqnarray}\label{f7}
\u_{x,i}=\prod_{y\in S(x)}(\cf(\u_{y}))_i, \ \ \textrm{for all} \ \
i\in\bn
\end{eqnarray}
where $\u_x=\{\u_{x,k}\}\in C$ for each $x\in\G_k$ is unknown
function and $\cf$ is given by \eqref{f1}.
 Since $c_0$ is an algebra, then
\eqref{f5} can be rewritten as follows
\begin{eqnarray}\label{f6}
\u_{x}=\prod_{y\in S(x)}\cf(\u_{y}).
\end{eqnarray}

According to Theorem \ref{Rt2} we conclude that the equation
\eqref{f6} has only one solution which is  $\u_x=\u_*$. Here $\u_*$
is a solution of the equation
$$
\big(\cf(\u)\big)^k=\u.
$$
Note that this solution $\u_*$ belongs to $C$ which follows from
Theorem \ref{Rt1}.

From this result, as a particular case, we obtain a main result of
the paper \cite{KM2009}, if one takes
$$
f_k(\xb)=p\sum_{j=1}^\infty x_j+1 \ \ \textrm{for all} \ \ k\in\bn,
$$
and $a=p(\theta-1)$, $b=\theta-1$, where $\theta\in\ce_p$.\\

Let $N\geq 2$ be a fixed natural number. Now consider another kind
of the functional equation
\begin{eqnarray}\label{f7}
\u_{x,i}=\prod_{y\in S(x)}\prod_{j=1}^N(\cf(\u_{y}))_{i+j}, \ \
\textrm{for all} \ \ i\in\bn
\end{eqnarray}
where as before $\u_x=\{\u_{x,k}\}\in C$, for each $x\in\G_k$, is
unknown function and $\cf$ is given by \eqref{f1}.

Let us rewrite the last equation in terms of elements of the algebra
$c_0$. Denote by $\sigma:c_0\to c_0$ the shift operator, i.e.
$$
(\s(\xb))_k=x_{k+1}, \ \ \ k\in\bn
$$
where $\xb=\{x_k\}\in c_0$. Then \eqref{f7} can be rewritten as
follows
\begin{eqnarray}\label{f8}
\u_{x}=\prod_{y\in S(x)}\prod_{j=1}^N\s^j(\cf(\u_{y})).
\end{eqnarray}

Again Theorem \ref{Rt2} implies the uniqueness of the solution of
\eqref{f6}, which is $\u_x=\u_*$ for all $x\in\G_k$. Here $\u_*$ is
a solution of the equation
$$
\bigg(\prod_{j=1}^N\s^j(\cf(\u))\bigg)^k=\u.
$$
Note that this solution $\u_*$ belongs to $C$ which follows from
Theorem \ref{Rt1}.

\section*{Acknowledgments}
The first named author (F.M.) acknowledges the Scientific and
Technological Research Council of Turkey (TUBITAK) for support, and
Zirve University (Gazinatep) for kind hospitality. F.M. also thanks
the MOHE grant ERGS13-024-0057, the IIUM grant EDW B13-029-0914 and
the Junior Associate scheme of the Abdus Salam International Centre
for Theoretical Physics, Trieste, Italy.

\end{document}